\documentclass{amsart}

\usepackage{amsmath,amsfonts,amssymb}

\usepackage{tikz}
\usepackage{multirow}
\usepackage{graphicx}
\usepackage{verbatim}
\usepackage{url}
\usepackage{hyperref}

\usepackage{float}
\floatstyle{plaintop}
\restylefloat{table}

\newtheorem{Thm}{Theorem}[section]
\newtheorem{Lem}[Thm]{Lemma}
\newtheorem{Prop}[Thm]{Proposition}

\newtheorem{Cor}[Thm]{Corollary}

\newcommand{\C}{\mathbb{C}}

\begin{document}

\newpage
\title{Explicit formulae for Chern-Simons invariants of the twist knot orbifolds and Edge polynomials of twist knots}
\author{Ji-Young Ham, Joongul Lee}
\address{Department of Science, Hongik University, 
94 Wausan-ro, Mapo-gu, Seoul,
 121-791\\
   Department of Mathematical Sciences, Seoul National University, 
   San 56-1 Shinrim-dong Kwanak-gu Seoul 151-747 \\
   Korea} 
\email{jiyoungham1@gmail.com.}

\address{Department of Mathematics Education, Hongik University, 
94 Wausan-ro, Mapo-gu, Seoul,
 121-791\\
   Korea} 
\email{jglee@hongik.ac.kr}



\maketitle 

\markboth{ Ji-young Ham, Joongul Lee } 
{ Chern-Simons invariants and edge polynomials }
 
\begin{abstract}
 We calculate the Chern-Simons invariants of the twist knot orbifolds using the Schl\"{a}fli formula for the generalized Chern-Simons function on the family of the twist knot cone-manifold structures. Following the general instruction of Hilden, Lozano, and Montesinos-Amilibia, we here present the concrete formulae and calculations. We use the Pythagorean Theorem, which was used by Ham, Mednykh, and Petrov, to relate the complex length of the longitude and the complex distance between the two axes fixed by two generators.
As an application, we  calculate the Chern-Simons invariants of cyclic coverings of the hyperbolic twist knot orbifolds.
 We also derive some interesting results. The explicit formula of $A$-polynomials of twist knots are obtained from the complex distance polynomials. Hence the edge polynomials corresponding to the edges of the Newton polygons of A-polynomials of twist knots can be obtained. In particular, the number of boundary components of every incompressible surface corresponding to slope $-4n+2$ appears to be $2$.
\end{abstract}
\maketitle

\section{Introduction}
In the 1970s, Chern and Simons~\cite{CS} defined an invariant of a compact $3\ (\text{mod}\ 4)$ - dimensional Riemannian manifold, M, which is now called the Chern-Simons invariant, $\text{\textnormal{cs}}(M)$. In the 1980s, Meyerhoff~\cite{Mey1} extended the definition of $\text{\textnormal{cs}}(M)$ to cusped manifolds. It is the integral of a certain 3-form and an invariant of the Riemannian connection on a principal tangent bundle of M. 

Various methods of finding the Chern-Simons invariant using ideal triangulations have been introduced~\cite{N1,N2,Z1,CMY1,CM1,CKK1} and implemented~\cite{SnapPy, Snap}.  But, for orbifolds, to our knowledge, there does not exist a single convenient program which computes Chern-Simons invariant.
The Chern-Simons invariant can also be obtained by eta-invariant, $\eta(M)$. $\text{\textnormal{cs}}(M)=\frac{3}{2} \eta(M)$ (mod  $\frac{1}{2}$)~\cite{CGHN1,Y1}. But it is easier to compute Chern-Simons invariants than $\eta$-invariants. 

Instead of working on complicated combinatorics of 3-dimensional ideal tetrahedra to find the Chern-Simons invariants of the twist knot orbifolds, we deal with simple one dimensional singular loci. To make the computation simpler, we express the complex length of the singular locus in terms of the complex distance between the two axes fixed by two generators. 
To find out the complex length of the singular locus, we start from working on $\text{\textnormal{SL}}(2,C)$. The singular locus appears in $\text{\textnormal{SL}}(2,C)$ as a series of matrix multiplications. If we first calculate the complex distance and recover the complex length of the singular locus from the complex distance using Lemma~\ref{lem:pytha}, the multiplication needed in $\text{\textnormal{SL}}(2,C)$ can be cut down approximately by half.
Similar methods for volumes can be found in~\cite{HMP}. We use the Schl\"{a}fli formula for the generalized Chern-Simons function on the family of a twist knot cone-manifold structures~\cite{HLM3}.  In~\cite{HLM2} a method of calculating the Chern-Simons invariants of two-bridge knot orbifolds were introduced but without explicit formulae. 
Similar approaches for $SU(2)$-connections can be found in~\cite{KK1} and for $\text{\textnormal{SL}}(2,C)$-connections in~\cite{KK2}.
Explicit integral formulae for Chern-Simons invariants of the Whitehead link (the two component twist link) orbifolds and their cyclic coverings are presented in~\cite{A,A1}.
A brief explanation for twist knot cone-manifolds are in~\cite{HMP}. You can also refer to~\cite{CHK,T1,K1,P2,HLM1,PW}.

\medskip

The main purpose of the paper is to find the explicit and efficient formulae for  Chern-Simons invariants of the twist knot orbifolds.  For two-bridge hyperbolic link, there exists an angle $\alpha_0 \in [\frac{2\pi}{3},\pi)$ for each link $K$ such that the cone-manifold $K(\alpha)$ is hyperbolic for $\alpha \in (0, \alpha_0)$, Euclidean for $\alpha=\alpha_0$, and spherical for $\alpha \in (\alpha_0, \pi]$ \cite{P2,HLM1,K1,PW}. We will use the Chern-Simons invariant of the lens space $L(4n+1,2n+1)$ calculated in~\cite{HLM2}.
The following theorem gives the formulae  for $T_m$ for even integers $m$. For odd integers $m$, we can replace $T_m$ by $T_{-m-1}$ as explained in Section~\ref{sec:twist}. So, the following theorem actually covers all possible hyperbolic twist knots. We exclude the non-hyperbolic case, $n=0,\ -1$.

\begin{Thm}\label{thm:main}
Let $T_{2n}$ be a hyperbolic twist knot. Let $T_{2n}(\alpha)$, $0 \leq \alpha < \alpha_0$ be a hyperbolic cone-manifold  with underlying space $S^3$ and with singular set $T_{2n}$ of cone-angle $\alpha$.  Let $k$ be a positive integer such that $k$-fold cyclic covering of $T_{2n}(\frac{2 \pi}{k})$ is hyperbolic. Then the Chern-Simons invariant of $T_{2n}(\frac{2 \pi}{k})$ (mod $\frac{1}{k}$ if $k$ is even or mod $\frac{1}{2k}$ if $k$ is odd) is given by the following formula:

\begin{equation*}
\begin{split}
\text{\textnormal{cs}} \left(T_{2n} \left(\frac{2 \pi}{k} \right)\right) & \equiv \frac{1}{2} \text{\textnormal{cs}}\left(L(4n+1,2n+1) \right) \\
&+\frac{1}{4 \pi^2}\int_{\frac{2 \pi}{k}}^{\alpha_0} Im \left(2*\log \left(M^{-2}\frac{A+iV}{A-iV}\right)\right) \: d\alpha \\
& +\frac{1}{4 \pi^2}\int_{\alpha_0}^{\pi}
 Im \left(\log \left(M^{-2}\frac{A+iV_1}{A-iV_1}\right)+\log \left(M^{-2}\frac{A+iV_2}{A-iV_2}\right)\right) \: d\alpha,
\end{split}
\end{equation*}

\noindent where 
for $A=\cot{\frac{\alpha}{2}}$, $V$ $(Im(V) \leq 0)$, $V_1$, and $V_2$ are zeroes of the complex distance polynomial $P_{2n}=P_{2n}(V,B)$ which is either given recursively by 

\medskip
\begin{equation*}
P_{2n} = \begin{cases}
 \left(\left(4 B^4-8 B^2+4\right) V^2-4 B^4+8 B^2-2\right) P_{2(n-1)} -P_{2(n-2)}, \ 
\text{if $n>1,$} \\
 \left(\left(4 B^4-8 B^2+4\right) V^2-4 B^4+8 B^2-2\right) P_{2(n+1)}-P_{2(n+2)},  \
\text{if $n<-1,$}
\end{cases}
\end{equation*}
\medskip

\noindent with initial conditions
\begin{equation*}
\begin{split}
P_{-2} (V,B) & =\left(2 B^2-2\right) V+2 B^2-1,\\
P_{0} (V,B) & = 1, \\
P_{2} (V,B) & =\left(4 B^4-8 B^2+4\right) V^2+\left(2-2 B^2\right) V-4 B^4+6 B^2-1 \\
\end{split}
\end{equation*}

\medskip

\noindent or given explicitly by 
\begin{equation*}
P_{2n} = \begin{cases}
 \sum_{i=0}^{2n} 
\left(\begin{smallmatrix}
n+ \lfloor \frac{i}{2} \rfloor \\
i
\end{smallmatrix}\right)
 \left(2 (B^2-1) (1-V)\right)^{i} 
\left(1+2 (V-1)^{-1}\right)^{\lfloor \frac{1+i}{2}  \rfloor}
\qquad
\text{if $n \geq 0$} \\
 \sum_{i=0}^{-2n-1} 
\left(\begin{smallmatrix}
-n+ \lfloor \frac{(i-1)}{2} \rfloor \\
i
\end{smallmatrix}\right)
 \left(2 (B^2-1) (V-1) \right)^i 
 \left(1+ 2 (V-1)^{-1}\right)^{\lfloor \frac{1+i}{2}  \rfloor}
\text{if $n<0$}
\end{cases}
\end{equation*}

\noindent where $B=\cos{\frac{\alpha}{2}}$ and $V_1$ and $V_2$ approach common $V$ as $\alpha$ decreases to $\alpha_0$ and they come from the components of $V$ and $\overline{V}$.
\end{Thm}

We here present some derived results. Theorem~\ref{thm:A-polynomial-re} gives the recursive formulae of A-polynomial of twist knots. In~\cite[Theorem 1]{HS}, Hoste and Shanahan presented the recursive formulae of A-polynomials of the twist knots with the opposite orientation.
 Theorem~\ref{thm:A-polynomial} gives the explicit formulae of A-polynomials of twist knots.  In~\cite[Theorem1.1]{Mat}, Mathews presented the explicit formulae of A-polynomials of the twist knots with the opposite orientation. In case $n \leq 0$ of ~\cite{Mat}, there is a typo; $\left(\frac{M^2-1}{L+M^2}\right)^i$ has to be changed into  $\left(\frac{1-M^2}{L+M^2}\right)^i$~\cite{Mat1}.
For each side of slop $a/b$ of the Newton polygon of $A_{2n}$, we can obtain the edge polynomial. By substituting t for $L^b M^a$ of each term appearing along the side edge, we have the edge polynomial, $f_{a/b}(t)$. 
Theorem~\ref{thm:edge-polynomial} gives the edge polynomials of twist knots. In~\cite[Corollary 11.5]{CL}, Cooper and Long showed that the edge polynomials of a two-bridge knot are all $\pm (t-1)^k  (t+1)^l$. We pin them down in case of twist knots. Corollary~\ref{cor:boundary} tells the number of boundary components in case of slope $-4n+2$ of twist knots. From~\cite{HT}, we know that the number of boundary components of  two-bridge knots are one or two. We pin them down in case of slope $-4n+2$ of twist knots.
Proofs of derived results are in Section~\ref{sec:poly}.

\begin{Thm}  \label{thm:A-polynomial-re}
A-polynomial $A_{2n}=A_{2n}(L,M)$ is given recursively by

\medskip
\begin{equation*}
A_{2n} = \begin{cases}
A_u A_{2(n-1)} 
 -M^4 \left(1+L M^2\right)^4 A_{2(n-2)} \ \ \ 
\text{if $n>1$} \\
 A_u A_{2(n+1)}
-M^4 \left(1+L M^2\right)^4 A_{2(n+2)}  \ \ \
\text{if $n<-2$}
\end{cases}
\end{equation*}
\medskip

\noindent with initial conditions
\begin{equation*}
\begin{split}
A_{-4} (L,M) & =1-L+2 L M^2+2 L M^4+L^2 M^4-L^2 M^6-L M^8+L M^{10}+2 L^2 M^{10}\\
&+2 L^2 M^{12}-L^2 M^{14}+L^3 M^{14},\\
A_{-2} (L,M) & =1+L M^6,\\
A_{0} (L,M) & = -1, \\
A_{2} (L,M) & =L-L M^2-M^4-2 L M^4-L^2 M^4-L M^6+L M^8, \\
\end{split}
\end{equation*}
where
\begin{equation*}
A_u=1-L+2 L M^2+M^4+2 L M^4+L^2 M^4+2 L M^6-L M^8+L^2 M^8.
\end{equation*}
\end{Thm}

\medskip

\begin{Thm}  \label{thm:A-polynomial}
A-polynomial $A_{2n}=A_{2n}(L,M)$ is given explicitly by 

\medskip
\begin{equation*}
A_{2n} = \begin{cases}
- M^{2 n} \left(1+L M^2\right)^{2 n} \sum_{i=0}^{2n} 
\left(\begin{smallmatrix}
n+ \lfloor \frac{i}{2} \rfloor \\
i
\end{smallmatrix}\right)
 \left(1-M^2\right)^i \left(1+L M^2\right)^{-i} \\
\times (L-1)^{\lfloor \frac{i}{2} \rfloor}\left(L M^2-M^{-2}\right)^{\lfloor \frac{1+i}{2}  \rfloor}
\qquad
\text{if $n \geq 0$} \\
 M^{-2 n} \left(1+L M^2\right)^{-2 n-1} \sum_{i=0}^{-2n-1} 
\left(\begin{smallmatrix}
-n+ \lfloor \frac{(i-1)}{2} \rfloor \\
i
\end{smallmatrix}\right)
 \left(M^2-1\right)^i \left(1+L M^2\right)^{-i} \\
\times (L-1)^{\lfloor \frac{i}{2} \rfloor}\left(L M^2-M^{-2}\right)^{\lfloor \frac{1+i}{2}  \rfloor}  \qquad
\text{if $n<0$}
\end{cases}
\end{equation*}
\end{Thm}

\medskip 

\begin{Thm}  \label{thm:edge-polynomial}
When $n > 0$, the edge polynomials of twist knots are
\begin{align*}
\begin{cases}
&\pm (t-1) \ \ \ \text{if the slope is $-4n$} \\
&-(t-1)^n  \ \ \ \text{if the slope is $4$ and n is even} \\
& \pm (t-1)^n    \ \ \ \text{if the slope is $4$ and n is odd}.
\end{cases}
\end{align*}

When $n < 0$, the edge polynomials of twist knots are
\begin{align*}
\begin{cases}
&t+1 \ \ \ \ \ \ \ \ \ \ \text{if the slope is $-4n+2$} \\
&\pm (t-1)^{-n-1}  \ \ \ \text{if the slope is $4$ and n is even} \\
& (t-1)^{-n-1}    \ \ \ \text{if the slope is $4$ and n is odd}.
\end{cases}
\end{align*}
\end{Thm}

\begin{Cor} \label{cor:boundary}
 The number of boundary components of every incompressible surface corresponding to slope $-4n+2$ of twist knots are $2$.
\end{Cor}

\section{Twist knots} \label{sec:twist}

\begin{figure} 
\resizebox{5cm}{!}{\includegraphics{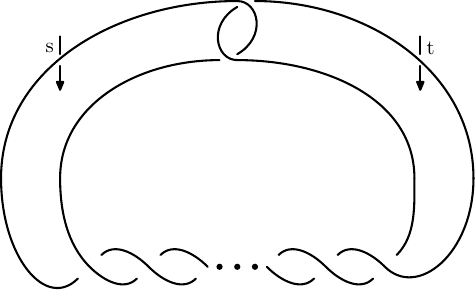}}
reflection
\reflectbox{\resizebox{5cm}{!}{\includegraphics{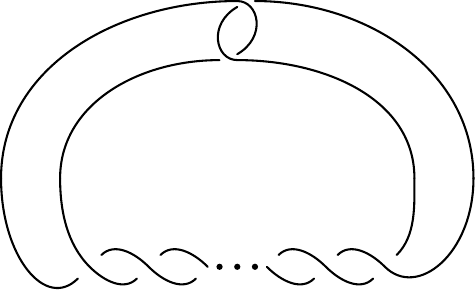}}}
\caption{twist knot (left) and its mirror image (right)} \label{fig:T2n}
\end{figure}

\begin{figure} 
\resizebox{4.5cm}{!}{\includegraphics{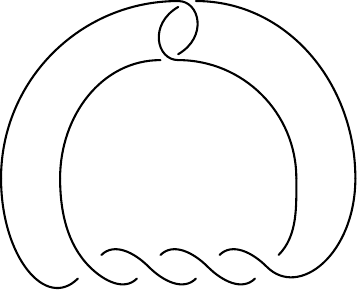}}
\ \ \
\resizebox{6cm}{!}{\includegraphics[angle=90]{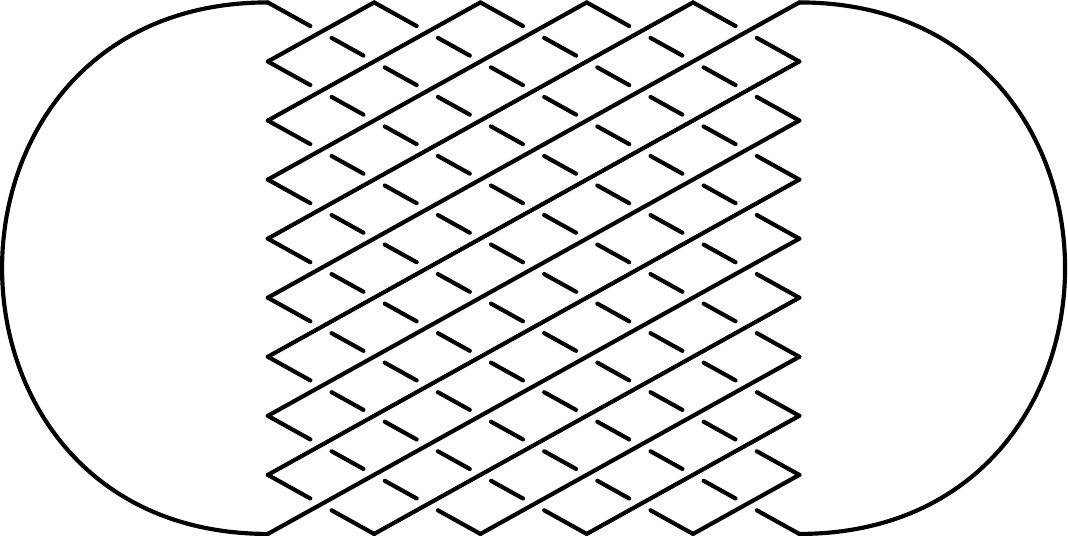}}
\caption{Knot $6_1$ with slope $2/9$ (left) and with slope $5/9$ (right).} \label{fig:knot}
\end{figure}

A knot $K$ is the twist knot if $K$ has a regular two-dimensional projection of the form in Figure~\ref{fig:T2n}. $K$ has 2 left-handed vertical crossings and $m$ right-handed horizontal crossings. We will denote it by $T_m$. One can easily check that the slope of $T_m$ is $2/(2m+1)$ which is equivalent to the knot with slope $(m+1)/(2m+1)$~\cite{S1}. For example, Figure~\ref{fig:knot}  shows two different regular projections of knot $6_1$; one with slope $2/9$ (left) and the other with slope $5/9$ (right).
Note that $T_m$ and its mirror image have the same fundamental group up to orientation and hence have the same fundamental domain up to isometry in $\mathbb{H}^3$. It follows that $T_m(\alpha)$ and its mirror image have the same fundamental set upto isometry in $\mathbb{H}^3$ and have the same Chern-Simons invariant up to sign. Since the mirror image of $T_m$ is equivalent to $T_{-m-1}$, when $m$ is odd we will use $T_{-m-1}$ for $T_m$. Hence a twist knot can be represented by $T_{2n}$ for some integer $n$ with slope $2/(4n+1)$ or $(2n+1)/(4n+1)$. 

Let us denote by $X_{2n}$ the exterior of $T_{2n}$. In~\cite{HS}, the fundamental group of $X_{2n}$ is calculated with 2 right-handed vertical crossings as positive crossings instead of two left-handed vertical crossings. The following theorem is tailored to our purpose. The following theorem can also be obtained by reading off the fundamental group from the Schubert normal form of $T_{2n}$ with slope $(2n+1)/(4n+1)$~\cite{R1}.

\begin{Prop}\label{prop:fundamentalGroup}
$$\pi_1(X_{2n})=\left\langle s,t \ |\ swt^{-1}w^{-1}=1\right\rangle,$$
where $w=(ts^{-1}t^{-1}s)^n$. 
\end{Prop}

\section{The complex distance polynomial and A-polynomial} \label{sec:poly}
Given a twist knot $T_{2n}$ and a set of generators, $\{s,t\}$, we identifiy the set of representations, $R$, of $\pi_1 (X_{2n})$ into $SL(2, \C)$ with $R\left(\pi_1 (X_{2n})\right)=\{(\eta(s),\eta(t)) \ |\ \eta \in R\}$. 
 Since the defining relation of 
 $\pi_1 (X_{2n})$ gives the defining equation of $R\left(\pi_1 (X_{2n})\right)$~\cite{R3}, $R\left(\pi_1 (X_{2n})\right)$ can be thought of an affine algebraic set in $\C^{2}$. 
$R\left(\pi_1 (X_{2n})\right)$ is well-defined upto isomorphisms which arise from changing the set of generators. We say elements in $R$ which differ by conjugations in $SL(2, \C)$ are equivalent. 

 We use two coordinates to give the structure of the affine algebraic set to $R\left(\pi_1 (X_{2n})\right)$. Equivalently, for some $O \in SL(2, \C)$, we consider both $\eta$ and $\eta^{\prime}=O^{-1} \eta O$:

For the complex distance polynomial, we use for the coordinates

$$\eta(s)=\left[\begin{array}{cc}
  ({M+1/M})/2    &  e^{\frac{\rho}{2}}   ({M-1/M})/2     \\
       e^{-\frac{\rho}{2}} ({M-1/M})/2    &  ({M+1/M})/2
                     \end{array} \right],$$

\medskip

$$\eta(t)=\left[\begin{array}{cc}
         ({M+1/M})/2 &  e^{-\frac{\rho}{2}} ({M-1/M})/2       \\
          e^{\frac{\rho}{2}} ({M-1/M})/2    &   ({M+1/M})/2
                 \end{array}  \right],$$

and for the A-polynomial, 
\begin{center}
$$\begin{array}{ccccc}
\eta^{\prime}(s)=\left[\begin{array}{cc}
                       M &       1 \\
                        0      & M^{-1}  
                     \end{array} \right]                          
\text{,} \ \ \
\eta^{\prime}(t)=\left[\begin{array}{cc}
                   M &  0      \\
                   t      & M^{-1} 
                 \end{array}  \right].
\end{array}$$
\end{center}

In~\cite{HMP}, the complex distance polynomial of $T_{2n}$ is presented recursively. Theorem~\ref{thm:cpolynomial-ex} gives it explicitly. It is the defining polynomial of algebraic set $R\left(\pi_1 (X_{2n})\right)$ with the set of generators given in Proposition~\ref{prop:fundamentalGroup} and with the coordinates $\eta(s)$ and $\eta(t)$ in $SL(2, \C)$.
The actual computation in~\cite{HMP} is done after setting $M=e^{\frac{i \alpha}{2}}$ and then the variables are changed to $B=\cos  \frac{\alpha}{2}$ and $V=\cosh \rho$.

\begin{Thm}~\cite[Theorem 3.1]{HMP} \label{thm:cpolynomial}
For $B=\cos  \frac{\alpha}{2}$, $V=\cosh \rho$ is a root of the following complex distance polynomial $P_{2n}=P_{2n}(V,B)$ which is given recursively by 

\medskip
\begin{equation*}
P_{2n} = \begin{cases}
 \left(\left(4 B^4-8 B^2+4\right) V^2-4 B^4+8 B^2-2\right) P_{2(n-1)} -P_{2(n-2)} \ 
\text{if $n>1$} \\
 \left(\left(4 B^4-8 B^2+4\right) V^2-4 B^4+8 B^2-2\right) P_{2(n+1)}-P_{2(n+2)}  \
\text{if $n<-1$}
\end{cases}
\end{equation*}
\medskip

\noindent with initial conditions
\begin{equation*}
\begin{split}
P_{-2} (V,B) & =\left(2 B^2-2\right) V+2 B^2-1,\\
P_{0} (V,B) & = 1, \\
P_{2} (V,B) & =\left(4 B^4-8 B^2+4\right) V^2+\left(2-2 B^2\right) V-4 B^4+6 B^2-1.\\
\end{split}
\end{equation*}
\end{Thm}

\begin{Thm} \label{thm:cpolynomial-ex}
For $B=\cos  \frac{\alpha}{2}$, $V=\cosh \rho$ is a root of the following complex distance polynomial $P_{2n}=P_{2n}(V,B)$ which is given explicitly by 

\medskip
\begin{equation*}
P_{2n} = \begin{cases}
 \sum_{i=0}^{2n} 
\left(\begin{smallmatrix}
n+ \lfloor \frac{i}{2} \rfloor \\
i
\end{smallmatrix}\right)
 \left(2 (B^2-1) (1-V)\right)^{i} 
\left(1+2 (V-1)^{-1}\right)^{\lfloor \frac{1+i}{2}  \rfloor}
\qquad \ \ \
\text{if $n \geq 0$} \\
 \sum_{i=0}^{-2n-1} 
\left(\begin{smallmatrix}
-n+ \lfloor \frac{(i-1)}{2} \rfloor \\
i
\end{smallmatrix}\right)
 \left(2 (B^2-1) (V-1) \right)^i 
 \left(1+ 2 (V-1)^{-1}\right)^{\lfloor \frac{1+i}{2}  \rfloor}  
\text{if $n<0$}. 
\end{cases}
\end{equation*}
\end{Thm}

\begin{proof}
We write $f_{2n}$ for the claimed formula and show that $f_{2n}=P_{2n}$.
$f_{2n}$ can be rewritten as
\begin{align*}
f_{2n}=\begin{cases} \sum_{j=0}^{n} 
\left(\begin{smallmatrix}
n+j \\
2j
\end{smallmatrix}\right) 
\left(2(B^2-1) (V-1)\right)^{2j} \left(1+2(V-1)^{-1} \right)^{j}\\
- \sum_{j=0}^{n} 
\left(\begin{smallmatrix}
n+j \\
2j+1
\end{smallmatrix}\right) 
\left(2(B^2-1) (V-1)\right)^{2j+1} \left(1+2(V-1)^{-1} \right)^{j+1} \ \text{if $n \geq 0$} \\
\sum_{j=0}^{-n-1} 
\left(\begin{smallmatrix}
-n-1+j \\
2j
\end{smallmatrix}\right) 
\left(2(B^2-1) (V-1)\right)^{2j} \left(1+2(V-1)^{-1} \right)^{j} \\
+ \sum_{j=0}^{-n-1} 
\left(\begin{smallmatrix}
-n+j \\
2j+1
\end{smallmatrix}\right) 
\left(2(B^2-1) (V-1)\right)^{2j+1} \left(1+2(V-1)^{-1} \right)^{j+1} \ \text{if $n < 0$}.
\end{cases}
\end{align*}

Now, the theorem follows by solving the recurrence formula with the initial conditions given in Theorem~\ref{thm:cpolynomial}:
\begin{equation*}
P_{2n} = \begin{cases}
 \left[\left(\left(2 (B^2-1) (V-1)\right)^2 -2 (B^2-1) (V-1) \right) \left(1+2(V-1)^{-1} \right)  +1\right]\\
\times h_{n-1}-h_{n-2} \qquad 
\text{if $n>1$} \\
 \left(2 \left(B^2-1\right) (V-1) \left(1+2(V-1)^{-1} \right)+1\right) h_{-n-1}-h_{-n-2}  \qquad
\text{if $n<-1$}
\end{cases}
\end{equation*}
\noindent where
\begin{align*}
h_{n}& =
 \sum_{k=0}^{\lfloor \frac{n}{2} \rfloor} 
\left(\begin{smallmatrix}
n+ 1 \\
2k+1
\end{smallmatrix}\right)
\left(2(B^2-1)^2 (V-1)^2 \left(1+2(V-1)^{-1} \right)+1\right)^{n-2k} \\
&\times \left(\left(2(B^2-1)^2 (V-1)^2\left(1+2(V-1)^{-1} \right)+1\right)^2-1\right)^{k}\\
 = & \sum_{i=0}^{\lfloor \frac{n}{2} \rfloor} 
\left(\begin{smallmatrix}
n- i \\
i
\end{smallmatrix}\right) (-1)^i
\left[ \sum_{j=0}^{n-2i} 
\left(\begin{smallmatrix}
n- 2i \\
j
\end{smallmatrix}\right) 2^{n-2i-j}
\left(\left(2(B^2-1) (V-1)\right)^2 \left(1+2(V-1)^{-1} \right)\right)^{j} \right] \\
= & \sum_{j=0}^{n} \left[\sum_{i=0}^{\lfloor \frac{n-j}{2} \rfloor} 
\left(\begin{smallmatrix}
n- i \\
i
\end{smallmatrix}\right) 
\left(\begin{smallmatrix}
n- 2i \\
j
\end{smallmatrix}\right) (-1)^i 2^{n-2i-j}\right]
\left(\left(2(B^2-1) (V-1)\right)^2 \left(1+2(V-1)^{-1} \right)\right)^{j} \\
= & \sum_{j=0}^{n} 
\left(\begin{smallmatrix}
n+1+j \\
2j+1
\end{smallmatrix}\right) 
\left(\left(2(B^2-1) (V-1)\right)^2 \left(1+2(V-1)^{-1} \right)\right)^{j}. 
\end{align*}
$f_{2n}$ can be obtained by simplifying the above formula.
\end{proof}

\medskip

\begin{proof}[proof of Theorem~\ref{thm:A-polynomial-re} and Theorem~\ref{thm:A-polynomial}]
Using Theorem 4.4 in~\cite{HMP}, we get the following  Lemma~\ref{lem:pytha} which relates the zeroes of the complex distance polynomial, $P_{2n}=P_{2n}(V,B)$, and the zeroes of the A-polynomial, $A_{2n}=A_{2n}(L,M)$. Theorem~\ref{thm:A-polynomial-re} (resp. Theorem~\ref{thm:A-polynomial}) can be obtained from $P_{2n}$ of Theorem~\ref{thm:cpolynomial} (resp. Theorem~\ref{thm:cpolynomial-ex}) by replacing $V$ with $\left((M^2+1) (LM^2-1)\right)/\left((M^2-1) (LM^2+1)\right)$  using the equality in Lemma~\ref{lem:pytha} and $B$ with $\left(M+M^{-1}\right)/2$ and by clearing denominators.
\end{proof}

\medskip

\begin{Lem}  \label{lem:pytha}
\begin{align}
iV&=A \frac{LM^2-1}{LM^2+1} \ \ \ and \ \ \ L=M^{-2} \frac{A+iV}{A-iV},\\ V&=\frac{(M^2+1) (LM^2-1)}{(M^2-1) (LM^2+1)} \ \ \ and \ \ \ L=M^{-2} \frac{VM^2-V+M^2+1}{-VM^2+V+M^2+1}.
\end{align}
\end{Lem}

\begin{proof}
$(1)$ is Theorem 4.4 in ~\cite{HMP}.

Since 
\begin{align*}
\frac{A}{i} &=\frac{\cos{\frac{\alpha}{2}}}{i\sin{\frac{\alpha}{2}}} 
 =\frac{\cosh{\frac{i\alpha}{2}}}{\sinh{\frac{i\alpha}{2}}} 
 =\frac{\frac{M+M^{-1}}{2}}{\frac{M-M^{-1}}{2}}
 =\frac{M^2+1}{M^2-1},
\end{align*}
we get the first equality of $(2)$. By solving the first equality for $L$, we get the second equality of $(2)$.
\end{proof}

\medskip

\begin{proof}[Proof of Theorem~\ref{thm:edge-polynomial}]
 By Lemma~\ref{lem:edge-polynomial}, we have Newton polygons, $NP$, associated to $A_{2n}$ in Figure~\ref{fig:NewtonPolygon}. Let us only consider nonzero slopes if $n \neq -1$.

 When $n >0$, $NP$ has two slopes $-4n$ and $4$. When the slope is $-4n$ and the edge has the term $-M^{4n}$ on it, there are only two terms of $A_{2n}$ appearing along the edge of the slope. From the explicit formula of $A_{2n}$ of Theorem~\ref{thm:A-polynomial}, they are $-M^{4n}$ and $L$. The term $-M^{4n}$ occurs when $i=2n$. We get the term $L$ by adding two terms $(n-1) L$ which occurs when $i=2n-1$ and $-nL$ which occurs when $i=2n$. From $L-M^{4n}$ by substituting $Lt$ for $M^{4n}$ and dividing by $L$, we get the edge polynomial $1-t$. Hence, on the right above edge with the same slope $-4n$, we have $t-1$ for the edge polynomial. When the slope is $4$ and the term $-M^{4n}$ is on the edge, from the explicit formula of $A_{2n}$ of Theorem~\ref{thm:A-polynomial}, 
$ n LM^{4n+4}$ is on the edge and the term occurs when $i=2n$. We can get the edge polynomial from the sum of the terms which appear on this edge by substituting $t$ for $LM^{4}$ and dividing by $M^{4n}$. Hence the coefficient of $LM^{4n+4}$ is the coefficient of $t$ of the edge polynomial. Since we know the constant term is $-1$ and the coefficient of $t$ is $n$, we get the edge polynomial $(t-1)^n$ when $n$ is odd and $-(t-1)^n$ when $n$ is even because the fact that the edge polynomials of two-bridge knots are up to sign the product of some powers of $t-1$ and some powers of $t+1$~\cite{CL} and the coefficient conditions forces the power of $t+1$ to be zero. Hence, on the right below edge with the same slope $4$, we have $-(t-1)^n$ for the edge polynomial.

 When $n <0$, $NP$ has two slopes $-4n+2$ and $4$. When the slope is $-4n+2$ and the edge has the term $1$ on it, there are only two terms of $A_{2n}$ appearing along the edge of the slope. From the explicit formula of $A_{2n}$ of Theorem~\ref{thm:A-polynomial}, they are $1$ and $LM^{-4n+2}$. The term $1$ occurs when $i=2n$. We get the term $LM^{-4n+2}$ by adding two terms $(n+1) LM^{-4n+2}$ which occurs when $i=-2n-2$ and $-nLM^{-4n+2}$ which occurs when $i=-2n-1$. From $1+LM^{-4n+2}$ by substituting $t$ for $LM^{-4n+2}$, we get the edge polynomial $t+1$. Hence, on the right above edge with the same slope $-4n+2$, we have $t+1$ for the edge polynomial. When the slope is $4$ and the term $LM^{-4n+2}$ is on the edge, from the explicit formula of $A_{2n}$ of Theorem~\ref{thm:A-polynomial}, 
$(n+1) L^2M^{-4n+6}$ is on the edge. We get the term $(n+1) L^2M^{-4n+6}$ by adding two terms $\frac{(n+1)(n+2)}{2} L^2M^{-4n+6}$ which occurs when $i=-2n-2$ and $\frac{-n (n+1)}{2} L^2M^{-4n+6}$ which occurs when $i=-2n-1$. We can get the edge polynomial from the sum of the terms which appear on this edge by substituting $t$ for $LM^{4}$ and dividing by $LM^{-4n+2}$. Hence the coefficient of $L^2M^{-4n+6}$ is the coefficient of $t$ of the edge polynomial. Since we know the constant term is $1$ and the coefficient of $t$ is $n+1$, we get the edge polynomial 
$(t-1)^{-n-1}$ when $n$ is odd and $-(t-1)^{-n-1}$ when $n$ is even because the fact that the edge polynomials of two-bridge knots are up to sign the product of some powers of $t-1$ and some powers $t+1$~\cite{CL} and the coefficient conditions forces the power of $t+1$ to be zero. Hence, on the right below edge with the same slope $4$, we have $(t-1)^{-n-1}$ for the edge polynomial.
\end{proof}

\medskip
\begin{Lem}~\label{lem:edge-polynomial}
The Newton polygons associated to $A_{2n}$ are polygons in Figure~\ref{fig:NewtonPolygon}.
\end{Lem}

\begin{proof}
The lemma is true when $n=-2$, $-1$, or $1$.

Since the Newton polygon of $A_{2n}$ has ones on the corners up to sign~\cite{CL2}, to determine the shape of the Newton polygon, we only need to consider $A_{2n}$ modulo $2$. We will use  the recursive formula of $A_{2n}$ of Theorem~\ref{thm:A-polynomial-re}. In modulo $2$, $A_u$ has $6$ terms and $M^4 \left(1+L M^2\right)^4$ has two terms $M^4+L^4M^{12}$. The lemma can be proved by induction. You just have to combine six copies of Newton polygons of $A_{2n-1}$ (if $n>1$) or $A_{2n+1}$ (if $n<-2$) shifted by $A_u$ and two copies of Newton polygons of $A_{2n-2}$ \ (if $n>1$) or $A_{2n+1}$ (if $n<-2$) shifted by $M^4+L^4M^{12}$  removing double points.
\end{proof}

\begin{figure} 
\includegraphics{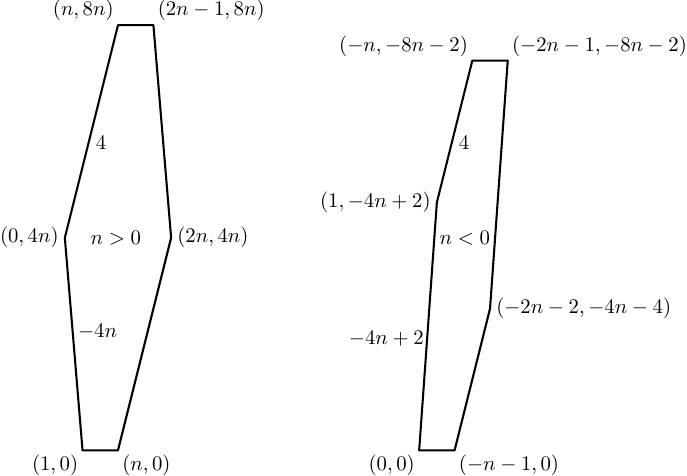}
\caption{Newton polygons of $A_{2n}$.} \label{fig:NewtonPolygon}
\end{figure}
\medskip

\begin{proof} [Proof of Corollary~\ref{cor:boundary}]
The number of boundary components of  two-bridge knots are one or two~\cite{HT}. Hence the number of boundary components of every incompressible surface corresponding to slope $-4n+2$ are bounded above by $2$. 

The orders of roots, roots of unity, of $f_{a/b}$ divide the number of boundary components of every incompressible surface corresponding to slope $a/b$~\cite{CCGLS1}. Hence, when the slope is $-4n+2$, since there is a single root of unity of degree $2$ by Theorem~\ref{thm:edge-polynomial}, the number of boundary components of every incompressible surface corresponding to slope $-4n+2$ is bounded below by $2$.
\end{proof}

\section{Generalized Chern-Simons function}
\label{sec:CSfunction}

The general references for this section are~\cite{HLM3,HLM2,Y1} and~\cite{MeyRub1}. 
We introduce the generalized Chern-Simons function on the family of a twist knot cone-manifold structures. For the oriented knot $T_{2n}$, we orient a chosen meridian $s$ such that the orientation of $s$ followed by orientation of $T_{2n}$ coincides with orientation of $S^3$. Hence, we use the definition of Lens space in~\cite{HLM2} so that we can have the right orientation when the definition of Lens space is combined with the following frame field. On the Riemannian manifold $S^3-T_{2n}-s$ we choose a special frame field $\Gamma$. A \emph{special} frame field $\Gamma=(e_1,e_2,e_3)$ is an orthonomal frame field such that for each point $x$ near $T_{2n}$, $e_1(x)$ has the knot direction, $e_2(x)$ has the tangent direction of a meridian curve, and $e_3(x)$ has the knot to point direction.  A special frame field always exists by the proposition $3.1$ of~\cite{HLM3}. From $\Gamma$ we obtain an orthonomal frame field $\Gamma_{\alpha}$ on $T_{2n}(\alpha)-s$ by the Schmidt orthonormalization process with respect to the geometric structure of the cone manifold $T_{2n}(\alpha)$. Moreover it can be made special by deforming it in a neighborhood of the singular set and $s$ if necessary. $\Gamma^{\prime}$ is an extention of $\Gamma$ to $S^3-T_{2n}$. For each cone-manifold $T_{2n}(\alpha)$, we assign the real number:

\begin{equation*}
I\left(T_{2n}(\alpha)\right)=\frac{1}{2} \int_{\Gamma(S^3-T_{2n}-s)}Q-\frac{1}{4 \pi} \tau(s,\Gamma^{\prime})-\frac{1}{4 \pi} \left(\frac{\beta \alpha}{2 \pi}\right),
\end{equation*}

\noindent where $-2 \pi \leq \beta \leq 2 \pi$, $Q$ is the Chern-Simons form:

\begin{equation*}
Q=\frac{1}{4 \pi^2} \left(\theta_{12} \wedge \theta_{13} \wedge \theta_{23} + \theta_{12} \wedge \Omega_{12} + \theta_{13} \wedge \Omega_{13} + \theta_{23} \wedge \Omega_{23} \right),
\end{equation*}

\noindent and 

\begin{equation*}
\tau(s,\Gamma^{\prime})=-\int_{\Gamma^{\prime}(s)} \theta_{23},
\end{equation*}

\noindent where ($\theta_{ij}$) is the connection $1$-form, ($\Omega_{ij}$) is the curvature $2$-form of the Riemannian connection on $T_{2n}(\alpha)$ and the integral is over the orthonomalizations of the same frame field. When $\alpha = \frac{2 \pi}{k}$ for some positive integer, 
$I \left(T_{2n}\left(\frac{2 \pi}{k}\right)\right)$ (mod $\frac{1}{k}$ if $k$ is even or mod $\frac{1}{2k}$ if $k$ is odd) is independent of the frame field $\Gamma$ and of the representative in the equivalence class $\overline{\beta}$ and hence an invariant of the orbifold $T_{2n}\left(\frac{2 \pi}{k}\right)$. $I \left(T_{2n}\left(\frac{2 \pi}{k}\right)\right)$ (mod $\frac{1}{k}$ if $k$ is even or mod $\frac{1}{2k}$ if $k$ is odd) is called \emph{the Chern-Simons invariant of the orbifold} and is denoted by 
$\text{\textnormal{cs}} \left(T_{2n}\left(\frac{2 \pi}{k}\right) \right)$.

On the generalized Chern-Simons function on the family of the twist knot cone-manifold structures we have the following Schl\"{a}fli formula.

\begin{Thm}(Theorem 1.2 of~\cite{HLM2})~\label{thm:schlafli}
For a family of geometric cone-manifold structures, $T_{2n}(\alpha)$, and differentiable functions $\alpha(t)$ and $\beta(t)$ of $t$ we have
\begin{equation*}
dI \left(T_{2n}(\alpha)\right)=-\frac{1}{4 \pi^2} \beta d \alpha.
\end{equation*} 
\end{Thm}

\section{Proof of the theorem~\ref{thm:main}} \label{sec:proof}

For $n \geq 1$ and $M=e^{i \frac{\alpha}{2}}$ ($B=\cos  \frac{\alpha}{2}$), $A_{2n}(M,L)$ and $P_{2n}(V,A)$ have $2n$ component zeros, and for  $n < -1$, $-(2n+1)$ component zeros. The component which gives the maximal volume is the geometric component~\cite{HMP,D1,FK1} and in~\cite{HMP} it is identified. For each $T_{2n}$, there exists an angle $\alpha_0 \in [\frac{2\pi}{3},\pi)$ such that $T_{2n}$ is hyperbolic for $\alpha \in (0, \alpha_0)$, Euclidean for $\alpha=\alpha_0$, and spherical for $\alpha \in (\alpha_0, \pi]$ \cite{P2,HLM1,K1,PW}.
Denote by $D(T_{2n}(\alpha))$ be the set of common zeros of the discriminant of $A_{2n}(L, e^{ \frac{i \alpha}{2}})$ over $L$  and the discriminant of
$P_{2n}(V,\cos  \frac{\alpha}{2})$ over $V$. Then $\alpha_0$ will be one of $D(T_{2n}(\alpha))$. 

 On the geometric component we can calculate
 the Chern-Simons invariant of an orbifold 
$T_{2n}(\frac{2 \pi}{k})$ (mod $\frac{1}{k}$ if $k$ is even or mod $\frac{1}{2k}$ if $k$ is odd), where $k$ is a positive integer such that $k$-fold cyclic covering of $T_{2n}(\frac{2 \pi}{k})$ is hyperbolic:
\begin{align*}
\text{\textnormal{cs}}\left(T_{2n} \left(\frac{2 \pi}{k} \right)\right) 
                       & \equiv I \left(T_{2n} \left(\frac{2 \pi}{k} \right)\right) 
                      \ \ \ \ \ \ \ \ \ \ \ \  \left(\text{mod} \ \frac{1}{k}\right) \\
                        & \equiv I \left(T_{2n}( \pi) \right)
                          +\frac{1}{4 \pi^2}\int_{\frac{2 \pi}{k}}^{\pi} \beta \: d\alpha 
                      \ \ \ \  \left(\text{mod} \ \frac{1}{k}\right) \\
                        & \equiv \frac{1}{2} \text{\textnormal{cs}}\left(L(4n+1,2n+1) \right) \\
&+\frac{1}{4 \pi^2}\int_{\frac{2 \pi}{k}}^{\alpha_0} Im \left(2*\log \left(M^{-2}\frac{A+iV}{A-iV}\right)\right) \: d\alpha \\
& +\frac{1}{4 \pi^2}\int_{\alpha_0}^{\pi}
 Im \left(\log \left(M^{-2}\frac{A+iV_1}{A-iV_1}\right)+\log \left(M^{-2}\frac{A+iV_2}{A-iV_2}\right)\right) \: d\alpha \\
& \left( \text{mod} \ \frac{1}{k}\ \text{if $k$ is even or }  \text{mod} \ \frac{1}{2k}\ \text{if $k$ is odd} \right)
\end{align*}
where the second equivalence comes from Theorem~\ref{thm:schlafli} and the third equivalence comes from the fact that $I \left(T_{2n}(\pi)\right) \equiv \frac{1}{2} \text{\textnormal{cs}}\left(L(4n+1,2n+1) \right)$  $\left(\text{mod }\frac{1}{2}\right)$, Lemma~\ref{lem:pytha}, and geometric interpretations of hyperbolic and spherical holonomy representations.

The fundamental set of the two-bridge link orbifolds are constructed in~\cite{MR2}. 
The following theorem gives the Chern-Simons invariant of the Lens space $L(4n+1,2n+1)$.

\begin{Thm}(Theorem 1.3 of~\cite{HLM2}) \label{thm:Lens}
\begin{align*}
\text{\textnormal{cs}} \left(L \left(4n+1,2n+1\right)\right) \equiv \frac{6n+4}{8n+2} && (\text{mod}\ 1).
\end{align*}
\end{Thm}

\section{Chern-Simons invariants of the hyperbolic twist knot orbifolds and of their cyclic coverings}

The table~\ref{table1-1} (resp. the table~\ref{table1-2}) gives the approximate Chern-Simons invariant of the hyperbolic twist knot orbifold, 
$\text{\textnormal{cs}} \left(T_{2n} (\frac{2 \pi}{k})\right)$ for $n$ between $2$ and $9$ (resp. for $n$ between $-9$ and $-2$) and for $k$ between $3$ and $10$, and of its cyclic covering, $cs \left(M_k (T_{2n})\right)$. We used Simpson's rule for the approximation with $10^4$ ($5 \times 10^3$ in Simpson's rule) intervals from $\frac{2 \pi}{k}$ to $\alpha_0$ and $10^4$ ($5 \times 10^3$ in Simpson's rule) intervals from $\alpha_0$ to $\pi$. 
The table~\ref{tab2} gives the approximate Chern-Simons invariant of 
$T_{2n}$ for each n between $-9$ and $9$ except the unknot, the torus knot, and the amphicheiral knot.  We again used Simpson's rule for the approximation with $10^4$ ($5 \times 10^3$ in Simpson's rule) intervals from $0$ to $\alpha_0$ and $10^4$ ($5 \times 10^3$ in Simpson's rule) intervals from $\alpha_0$ to $\pi$. 
We used Mathematica for the calculations. We record here that our data in table~\ref{tab2} and those obtained from  SnapPy match up up to six decimal points.

\begin{table}
\begin{tabular}{ll}
\begin{tabular}{|c|c|c|}
\hline
 $k$ & $\text{\textnormal{cs}} \left(T_4(\frac{2 \pi}{k})\right)$ & $\text{\textnormal{cs}} \left(M_k( T_4)\right)$ \\
\hline
 3 & 0.0875301 \text{} & 0.262590 \text{} \\
 4 & 0.144925 & 0.579699 \\
 5 & 0.0784576 \text{} & 0.392288 \text{} \\
 6 & 0.0351571 & 0.210943 \\
 7 & 0.00506505 \text{} & 0.0354553 \text{} \\
 8 & 0.108039 & 0.864313 \\
 9 & 0.0218112 \text{} & 0.196301 \text{} \\
 10 & 0.0530574 & 0.530574 \\
\hline
\end{tabular}
&
\begin{tabular}{|c|c|c|}
\hline
 $k$ &  $\text{\textnormal{cs}} \left(T_6(\frac{2 \pi}{k})\right)$ & $\text{\textnormal{cs}} \left(M_k (T_6)\right)$ \\
\hline
 3 & 0.0449535 \text{} & 0.134860 \text{} \\
 4 & 0.0876043 & 0.350417 \\
 5 & 0.0165337 \text{} & 0.0826684 \text{} \\
 6 & 0.138167 & 0.829004 \\
 7 & 0.0120078 \text{} & 0.0840545 \text{} \\
 8 & 0.0430876 & 0.344700 \\
 9 & 0.0121250 \text{} & 0.109125 \text{} \\
 10 & 0.0876213 & 0.876213 \\
\hline
\end{tabular}
\end{tabular}

\bigskip

\begin{tabular}{ll}
\begin{tabular}{|c|c|c|}
\hline
 $k$ &   $\text{\textnormal{cs}} \left(T_8(\frac{2 \pi}{k})\right)$ & $\text{\textnormal{cs}} \left(M_k (T_8)\right)$ \\
\hline
 3 & 0.0161266 \text{} & 0.0483799 \text{} \\
 4 & 0.0536832 & 0.214733 \\
 5 & 0.0817026 \text{} & 0.408513 \text{} \\
 6 & 0.103012 & 0.618074 \\
 7 & 0.0481239 \text{} & 0.336867 \text{} \\
 8 & 0.00768503 & 0.0614802 \\
 9 & 0.0322210 \text{} & 0.289989 \text{} \\
 10 & 0.0521232 & 0.521232 \\
\hline
\end{tabular}
&
\begin{tabular}{|c|c|c|}
\hline
 $k$ &  $\text{\textnormal{cs}} \left(T_{10} (\frac{2 \pi}{k})\right)$ & $\text{\textnormal{cs}} \left( M_k (T_{10})\right)$ \\
\hline
3 & 0.162697 \text{} & 0.488091 \text{} \\
 4 & 0.0320099 & 0.128040 \\
 5 & 0.0597580 \text{} & 0.298790 \text{} \\
 6 & 0.0809665 & 0.485799 \\
 7 & 0.0260276 \text{} & 0.182193 \text{} \\
 8 & 0.110559 & 0.884475 \\
 9 & 0.0100766 \text{} & 0.0906893 \text{} \\
 10 & 0.0299660 & 0.299660 \\
\hline
\end{tabular}
\end{tabular}

\bigskip

\begin{tabular}{ll}
\begin{tabular}{|c|c|c|}
\hline
 $k$ &  $\text{\textnormal{cs}} \left(T_{12} (\frac{2 \pi}{k})\right)$ & $\text{\textnormal{cs}} \left(M_k (T_{12})\right)$ \\
\hline
 3 & 0.148360 \text{} & 0.445081 \text{} \\
 4 & 0.0170833 & 0.0683334 \\
 5 & 0.0447221 \text{} & 0.223610 \text{} \\
 6 & 0.0658884 & 0.395330 \\
 7 & 0.0109281 \text{} & 0.0764969 \text{} \\
 8 & 0.0954474 & 0.763579 \\
 9 & 0.0505121 \text{} & 0.454609 \text{} \\
 10 & 0.0148406 & 0.148406 \\
\hline
\end{tabular}
&
\begin{tabular}{|c|c|c|}
\hline
 $k$ &  $\text{\textnormal{cs}} \left(T_{14} (\frac{2 \pi}{k})\right)$ & $\text{\textnormal{cs}} \left(M_k (T_{14})\right)$ \\
\hline
 3 & 0.137750 \text{} & 0.413249 \text{} \\
 4 & 0.00620422 & 0.0248169 \\
 5 & 0.0337900 \text{} & 0.168950 \text{} \\
 6 & 0.0549355 & 0.329613 \\
 7 & 0.0713932 \text{} & 0.499752 \text{} \\
 8 & 0.0844779 & 0.675823 \\
 9 & 0.0395376 \text{} & 0.355839 \text{} \\
 10 & 0.00386414 & 0.0386414 \\
\hline
\end{tabular}
\end{tabular}

\bigskip

\begin{tabular}{ll}
\begin{tabular}{|c|c|c|}
\hline
 $k$ &  $\text{\textnormal{cs}} \left(T_{16} (\frac{2 \pi}{k})\right)$ & $\text{\textnormal{cs}} \left(M_k (T_{16})\right)$ \\
\hline
 3 & 0.129617 \text{} & 0.388850 \text{} \\
 4 & 0.247931 & 0.991725 \\
 5 & 0.0254880 \text{} & 0.127440 \text{} \\
 6 & 0.0466221 & 0.279733 \\
 7 & 0.0630741 \text{} & 0.441518 \text{} \\
 8 & 0.0761526 & 0.609221 \\
 9 & 0.0312096 \text{} & 0.280887 \text{} \\
 10 & 0.0955375 & 0.955375 \\
\hline
\end{tabular}
&
\begin{tabular}{|c|c|c|}
\hline
 $k$ &  $\text{\textnormal{cs}} \left(T_{18} (\frac{2 \pi}{k})\right)$ & $\text{\textnormal{cs}} \left(M_k (T_{18})\right)$ \\
\hline
3 & 0.123198 \text{} & 0.369593 \text{} \\
 4 & 0.241431 & 0.965725 \\
 5 & 0.0189709 \text{} & 0.0948543 \text{} \\
 6 & 0.0400981 & 0.240588 \\
 7 & 0.0565432 \text{} & 0.395803 \text{} \\
 8 & 0.0696194 & 0.556955 \\
 9 & 0.0246862 \text{} & 0.222176 \text{} \\
 10 & 0.0890057 & 0.890057 \\
\hline
\end{tabular}
\end{tabular}

\bigskip

\caption{Chern-Simons invariant of the hyperbolic twist knot orbifold, $\text{\textnormal{cs}} \left(T_{2n} (\frac{2 \pi}{k})\right)$ for $n$ between $2$ and $9$ and for $k$ between $3$ and $10$, and of its cyclic covering, $\text{\textnormal{cs}} \left(M_k (T_{2n})\right)$.}
\label{table1-1}
\end{table}
\begin{table}
\begin{tabular}{ll}
\begin{tabular}{|c|c|c|}
\hline
 $k$ &  $\text{\textnormal{cs}} \left(T_{-4} (\frac{2 \pi}{k})\right)$ & $\text{\textnormal{cs}} \left(M_k (T_{-4})\right)$ \\
\hline
3 & 0.0200144 \text{} & 0.0600431 \text{} \\
 4 & 0.186811 & 0.747246 \\
 5 & 0.00166667 \text{} & 0.00833333 \text{} \\
 6 & 0.0504622 & 0.302773 \\
 7 & 0.0163442 \text{} & 0.114410 \text{} \\
 8 & 0.116990 & 0.935921 \\
 9 & 0.0292902 \text{} & 0.263612 \text{} \\
 10 & 0.0595432 & 0.595432 \\
\hline
\end{tabular}

&
\begin{tabular}{|c|c|c|}
\hline
 $k$ &  $\text{\textnormal{cs}} \left(T_{-6}(\frac{2 \pi}{k})\right)$ & $\text{\textnormal{cs}} \left(M_k (T_{-6})\right)$ \\
\hline
3 & 0.0749433 \text{} & 0.224830 \text{} \\
 4 & 0.0126376 & 0.0505506 \\
 5 & 0.0873477 \text{} & 0.436738 \text{} \\
 6 & 0.140792 & 0.844753 \\
 7 & 0.0376998 \text{} & 0.263898 \text{} \\
 8 & 0.0862114 & 0.689691 \\
 9 & 0.0133130 \text{} & 0.119817 \text{} \\
 10 & 0.0552937 & 0.552937 \\
\hline
\end{tabular}
\end{tabular}

\bigskip

\begin{tabular}{ll}
\begin{tabular}{|c|c|c|}
\hline
 $k$ &  $\text{\textnormal{cs}} \left(T_{-8}(\frac{2 \pi}{k})\right)$ & $\text{\textnormal{cs}} \left(M_k (T_{-8})\right)$ \\
\hline
 3 & 0.109659 \text{} & 0.328978 \text{} \\
 4 & 0.0564153 & 0.225661 \\
 5 & 0.0330919 \text{} & 0.165460 \text{} \\
 6 & 0.0205610 & 0.123366 \\
 7 & 0.0130371 \text{} & 0.0912595 \text{} \\
 8 & 0.00816423 & 0.0653138 \\
 9 & 0.00482762 \text{} & 0.0434486 \text{} \\
 10 & 0.00244289 & 0.0244289 \\
\hline
\end{tabular}
&
\begin{tabular}{|c|c|c|}
\hline
 $k$ &  $\text{\textnormal{cs}} \left(T_{-10}(\frac{2 \pi}{k})\right)$ & $\text{\textnormal{cs}} \left(M_k (T_{-10})\right)$ \\
\hline
3 & 0.133696 \text{} & 0.401087 \text{} \\
 4 & 0.0832469 & 0.332988 \\
 5 & 0.0603968 \text{} & 0.301984 \text{} \\
 6 & 0.0480386 & 0.288232 \\
 7 & 0.0405999 \text{} & 0.284200 \text{} \\
 8 & 0.0357763 & 0.286210 \\
 9 & 0.0324710 \text{} & 0.292239 \text{} \\
 10 & 0.0301075 & 0.301075 \\
\hline
\end{tabular}
\end{tabular}

\bigskip

\begin{tabular}{ll}
\begin{tabular}{|c|c|c|}
\hline
 $k$ &  $\text{\textnormal{cs}} \left(T_{-12}(\frac{2 \pi}{k})\right)$ & $\text{\textnormal{cs}} \left(M_k (T_{-12})\right)$ \\
\hline
 3 & 0.150597 \text{} & 0.451792 \text{} \\
 4 & 0.101087 & 0.404347 \\
 5 & 0.0784041 \text{} & 0.392020 \text{} \\
 6 & 0.0661095 & 0.396657 \\
 7 & 0.0587029 \text{} & 0.410920 \text{} \\
 8 & 0.0538979 & 0.431183 \\
 9 & 0.0506045 \text{} & 0.455441 \text{} \\
 10 & 0.0482492 & 0.482492 \\
\hline
\end{tabular}

&
\begin{tabular}{|c|c|c|}
\hline
 $k$ &  $\text{\textnormal{cs}} \left(T_{-14}(\frac{2 \pi}{k})\right)$ & $\text{\textnormal{cs}} \left(M_k (T_{-14})\right)$ \\
\hline
 3 & 0.162874 \text{} & 0.488621 \text{} \\
 4 & 0.113753 & 0.455011 \\
 5 & 0.0911448 \text{} & 0.455724 \text{} \\
 6 & 0.0788794 & 0.473276 \\
 7 & 0.0000589596 \text{} & 0.000412717 \text{} \\
 8 & 0.0666914 & 0.533532 \\
 9 & 0.00784780 \text{} & 0.0706302 \text{} \\
 10 & 0.0610519 & 0.610519 \\
\hline
\end{tabular}
\end{tabular}

\bigskip

\begin{tabular}{ll}
\begin{tabular}{|c|c|c|}
\hline
$k$ &  $\text{\textnormal{cs}} \left(T_{-16}(\frac{2 \pi}{k})\right)$ & $\text{\textnormal{cs}} \left(M_k (T_{-16})\right)$ \\
\hline
3 & 0.00545933 \text{} & 0.0163780 \text{} \\
 4 & 0.123196 & 0.492785 \\
 5 & 0.000626850 \text{} & 0.00313425 \text{} \\
 6 & 0.0883767 & 0.530260 \\
 7 & 0.00956397 \text{} & 0.0669478 \text{} \\
 8 & 0.0762009 & 0.609607 \\
 9 & 0.0173633 \text{} & 0.156270 \text{} \\
 10 & 0.0705667 & 0.705667 \\
\hline
\end{tabular}
&
\begin{tabular}{|c|c|c|}
\hline
$k$ &  $\text{\textnormal{cs}} \left(T_{-18}(\frac{2 \pi}{k})\right)$ & $\text{\textnormal{cs}} \left(M_k (T_{-18})\right)$ \\
\hline
 3 & 0.0126607 \text{} & 0.0379822 \text{} \\
 4 & 0.130503 & 0.522012 \\
 5 & 0.00795573 \text{} & 0.0397786 \text{} \\
 6 & 0.0957144 & 0.574286 \\
 7 & 0.0169085 \text{} & 0.118359 \text{} \\
 8 & 0.0835496 & 0.668397 \\
 9 & 0.0247059 \text{} & 0.222353 \text{} \\
 10 & 0.0779144 & 0.779144 \\
\hline
\end{tabular}
\end{tabular}
\caption{Chern-Simons invariant of the hyperbolic twist knot orbifold, $\text{\textnormal{cs}} \left(T_{2n} (\frac{2 \pi}{k})\right)$ for $n$ between $-9$ and $-2$ and for $k$ between $3$ and $10$, and of its cyclic covering, $\text{\textnormal{cs}} \left(M_k (T_{2n})\right)$.}
\label{table1-2}
\end{table}
\begin{table} 
\begin{tabular}{cc} 
\begin{tabular}{|c|c|c|}
\hline
2n & $\alpha_0$ & $\text{\textnormal{cs}}\left(T_{2n}\right)$  \\
\hline
 4 & 2.57414 & 0.344023 \text{} \\
 6 & 2.75069 & 0.277867 \text{} \\
 8 & 2.84321 & 0.242222 \text{} \\
 10 & 2.90026 & 0.220016 \text{} \\
 12 & 2.93897 & 0.204869 \text{} \\
 14 & 2.96697 & 0.193882 \text{} \\
 16 & 2.98817 & 0.185550 \text{} \\
 18 & 3.00477 & 0.179014 \text{} \\
\hline
\end{tabular}
&
\begin{tabular}{|c|c|c|}
\hline
2n  & $\alpha_0$ & $\text{\textnormal{cs}}\left(T_{2n}\right)$  \\
\hline
-4 & 2.40717 & 0.346796 \text{} \\
 -6 & 2.67879 & 0.444846 \text{} \\
 -8 & 2.80318 & 0.492293 \text{} \\
 -10 & 2.87475 & 0.0200385 \text{} \\
 -12 & 2.92130 & 0.0382117 \text{} \\
 -14 & 2.95401 & 0.0510293 \text{} \\
 -16 & 2.97825 & 0.0605519 \text{} \\
 -18 & 2.99694 & 0.0679043 \text{} \\
\hline
\end{tabular}
\end{tabular}
\caption{Chern-Simons invariant of $T_{2n}$ for $n$ between $2$ and $9$ and for $n$ between $-9$ and $-2$).}
\label{tab2}
\end{table}
\medskip

\emph{Acknowledgements.} The authors would like to thank Alexander Mednykh and Hyuk Kim for their various helps, Nathan Dunfield and Daniel Mathews for their prompt helps and anonymous referees.


\end{document}